\documentclass[12pt]{article}
\usepackage{fullpage,amsfonts,amsmath,amsthm,amssymb,mathrsfs,graphicx}

\usepackage[top=1in, bottom=1in, left=.75in, right=.75in]{geometry}

\newcommand{\lab}[1]{\label{#1}}                

\newcommand{\remove}[1]{}
\newcommand\eqn[1]{(\ref{#1})}

\newcommand{\be}{\begin{equation}}
\newcommand{\ee}{\end{equation}}
\newcommand{\bea}{\begin{eqnarray}}
\newcommand{\eea}{\end{eqnarray}}
\newcommand{\bean}{\begin{eqnarray*}}
\newcommand{\eean}{\end{eqnarray*}}

\newtheorem{thm}{Theorem}
\newtheorem{cor}[thm]{Corollary}

\newtheorem{lemma}[thm]{Lemma}
\newtheorem{definition}[thm]{Definition}

\newtheorem{prop}[thm]{Proposition}
\newtheorem{obs}[thm]{Observation}
\newtheorem{prob}[thm]{Problem}

\def \T{{\mathcal T}}

\def\no{\noindent}
\def\S{{\mathcal S}}

\date{}

\title{On the geometric Ramsey numbers of trees}
\author{Pu Gao\thanks{Research supported by the NSERC PDF fellowship}\\ University of Toronto, Toronto, ON, M5S 3G4, Canada\\
pu.gao@utoronto.ca}
\begin{document}
\maketitle

\begin{abstract}
In this paper, we obtain upper bounds for the geometric Ramsey numbers of trees. We prove that $R_c(T_n,H_m)=(n-1)(m-1)+1$ if $T_n$ is a caterpillar and $H_m$ is a Hamiltonian outerplanar graph on $m$ vertices. Moreover, if $T_n$ has  at most two non-leaf vertices, then $R_g(T_n,H_m)=(n-1)(m-1)+1$. We also prove that $R_c(T_n,H_m)=O(n^2m)$ and $R_g(T_n,H_m)=O(n^3m^2)$ if $T_n$ is an arbitrary tree on $n$ vertices and $H_m$ is an outerplanar triangulation with pathwidth 2. 
\end{abstract}



\section{Introduction}

Given a graph $G$, the Ramsey number of $G$, denoted by $R(G)$, is the minimum integer $n$ such that for any $2$-colouring of the edges of the complete graph $K_n$, there is a monochromatic subgraph isomorphic to $G$. The geometric Ramsey theory combines the Ramsey theory with combinatorial geometry. A set $P$ of points in the plane is in general position if no three points in $P$ are collinear. A geometric graph on $P$ is the graph obtained by taking the vertex set to be the points in $P$ and the edge set to be a subset of the line segments between any two points in $P$. We denote by $K_P$ the complete geometric graph on $P$. Given an outerplaner graph $G$ (a graph that has a planar embedding such that every vertex is incident with the outer face), the general geometric Ramsey number of $G$, denoted by $R_g(G)$, is the minimum integer $n$, such that for any set $P$ of $n$ points in the plane in general position and for any $2$-edge-colouring of $K_P$, there is a monochromatic non-crossing subgraph isomorphic to $G$. (Here non-crossing means no two edges cross each other.) The convex geometric Ramsey number of $G$, denoted by $R_c(G)$, is defined in the same manner, except that the points in $P$ are in convex position. Clearly, $R_c(G)$ and $R_g(G)$ are well defined only if $G$ is outerplaner. In general, determining the exact values, or even determining the correct asymptotic orders, of the geometric Ramsey numbers is difficult for most outerplanar graphs. There are only a small class of graphs whose geometric Ramsey numbers are exactly determined or well estimated. See~\cite{KPT,KPTV}.  Indeed, whether there is a uniform polynomial upper bound for $R_c(G)$ and $R_g(G)$ for all outerplanar graphs $G$ is not known. It is not even known whether there is a uniform polynomial upper bound for the geometric Ramsey numbers of trees. Trivially, we have $R_c(G)\le R_g(G)$ always, even though it is not known yet whether there exists an outerplanar graph $G$ for which $R_c(G)$  differs from $R_g(G)$. A trivial upper bound (with an exponential order) of $R_g(G)$ is $R(K_n)$, where $n=|V(G)|$, and $K_n$ is the complete graph on $n$ vertices, due to the following (rephrased) theorem.

\begin{thm}[Gritzmann et al.~\cite{GMPP} ]\lab{thm:complete}
Let $n\ge 1$. For every outerplanar graph $G$ on $n$ vertices and for any set $P$ of $n$ points in the plane in general position, $K_P$ contains a non-crossing subgraph isomorphic to $G$. Moreover, if $v$ is an arbitrary vertex in $G$ and $p$ is an arbitrary point on the convex hull of $P$, then there is a non-crossing embedding of $G$ in $K_P$ so that $v$ is embedded at $p$.
\end{thm}

In this paper, all colourings will refer to the edge colouring unless otherwise specified and a $2$-colouring of a (geometric) graph refers to a colouring of edges with the blue and red colours. Given two graphs $G_1$ and $G_2$, the (nonsymmetric) Ramsey number $R(G_1,G_2)$ is defined to be the minimum integer $n$ such that for any $2$-colouring of $K_n$, either there is a red subgraph isomorphic to $G_1$, or there is a blue subgraph isomorphic to $G_2$. We define $R_g(G_1,G_2)$ similarly for outerplanar graphs $G_1$ and $G_2$ but the geometric restriction of non-crossing is applied to both $G_1$ and $G_2$.
The following classical result on the Ramsey numbers of cliques versus trees is due to Chv\'{a}tal~\cite{C}.
\begin{thm}\lab{thm:Chvatal}
For every $m\ge 1$ and $n\ge 1$,  $
R(T_n,K_m)=(m-1)(n-1)+1$, where $T_n$ is an arbitrary tree on $n$ vertices.
\end{thm}
A key observation in the proof of Theorem~\ref{thm:Chvatal} is the following lemma.
\begin{lemma}\lab{lem:Chvatal}
Let $m\ge 1$ and $n\ge 1$. For any graph $G$ on $N$ vertices, where $N=(m-1)(n-1)+1$, either $G$ contains a subgraph whose minimum degree is at least $n-1$ or the complement of $G$ contains $K_m$ as a subgraph.
\end{lemma}
The concept in Theorem~\ref{thm:complete} generalises to geometric Ramsey theory. In the 2nd Eml\'{e}kt\'{a}bla Workshop, Gy. K\'{a}rolyi asked whether Theorem~\ref{thm:Chvatal} is true if $R(T_n,K_m)$ is replaced by $R_c(T_n,H_m)$ or $R_g(T_n,H_m)$, where $T_n$ is a tree with $n$ vertices and $H_m$ is any Hamiltonian outerplanar graph with $m$ vertices. A formal statement of the problem  can be found in~\cite[Problem 6]{G} and is rephrased as follows.

\begin{prob}\lab{prob:clique-tree}
Let $m\ge 1$ and $n\ge 1$ be integers and let $T_n$ be an arbitrary tree on $n$ vertices and $H_m$ be an arbitrary Hamiltonian outerplanar graph on $m$ vertices. Is it true that $R_c(T_n,H_m)=R_g(T_n,H_m)=(n-1)(m-1)+1$?
\end{prob}

It is easy to see that for any such $T_n$ and $H_m$,
\begin{equation}
R_g(T_n,H_m)\ge R_c(T_n,H_m)\ge (n-1)(m-1)+1 \lab{lowerbound}
\end{equation}
by the following construction, which can be found in~\cite{KR}. Let $P$ be a set of $(n-1)(m-1)$ points in convex position. Pick an arbitrary point and label it $v_1$. Label the other vertices $v_i$, $2\le i\le (n-1)(m-1)$ so that $v_1$ sees $v_i$ before $v_{i+1}$ in the clockwise order, for each $i$. Partition $P$ into $S_1,\ldots,S_{m-1}$ such that for each $1\le j\le m-1$, $S_j=\{v_{i}:\ (j-1)(n-1)+1\le i\le j(n-1)\}$. For each edge $x=\{u,v\}$, if $u$ and $v$ are in the same part, colour $x$ red; otherwise, colour it blue. Clearly, the subgraph of $K_P$ induced by the red edges is a union of $m-1$ disjoint copies of $K_{n-1}$, and so $K_P$ does not contain a red copy of $T_n$. On the other hand, the longest blue non-crossing cycle in $K_P$ has length $m-1$ by our construction and so $K_P$ contains no blue non-crossing copy of $H_m$ since $H_m$ is Hamiltonian. This implies that $R_g(T_n,H_m)\ge R_c(T_n,H_m)\ge (n-1)(m-1)+1$.

The goal of this paper is to provide estimates of $R_c(T_n,H_m)$ and $R_g(T_n,H_m)$ in Problem~\ref{prob:clique-tree}.

\subsection{Main results}

 The answer to Problem~\ref{prob:clique-tree} is yes if $T_n$ is a path, following from Theorems~\ref{thm:path-clique} (see Section~\ref{sec:caterpillar}) and~\ref{thm:complete} and it is also yes if $T_n$ is a star, following from Theorem~\ref{thm:Chvatal} and by noting that stars are never self-crossing. A {\em caterpillar} is a tree such that the subgraph induced by all non-leaf vertices is a path.  We confirm with a positive answer to Problem~\ref{prob:clique-tree} for all caterpillars in the convex case and for trees containing exactly two non-leaf vertices in the general case.

\begin{thm}\lab{thm:dumbbell} Let $n\ge 1$ and $m\ge 1$. Suppose $T_n$ is a caterpillar on $n$ vertices and $H_m$ is an ourterplanar graph on $m$ vertices which contains a Hamilton cycle. Then,
\begin{enumerate}
\item[(i)] $R_c(T_n,H_m)= (n-1)(m-1)+1$.
\item[(ii)] $R_g(T_n,H_m)\le \Delta \gamma m^2$, where $\Delta$ denotes the maximum degree of $T_n$ and $\gamma$ denotes the number of non-leaf vertices of $T_n$.
\item[(iii)] Moreover, if $T_n$ has exactly two non-leaf vertices, then $R_g(T_n,H_m)= (n-1)(m-1)+1$.
\end{enumerate}
\end{thm}

An outerplanar triangulation is a graph that has a planar embedding such that all vertices are incident with the outer face and all inner faces are incident with exactly three edges.
Given a graph $G$, a sequence of subsets of $V(G)$, denoted by $X_1,X_2,\ldots,X_k$, is called a path-decomposition of $G$, if
\begin{enumerate}
\item[(a)] For every edge $e=\{u,v\}$ in $G$, there is an $X_i$ such that $u,v\in X_i$;
\item[(b)] For every vertex $u\in G$, the set of $X_i$'s that contain $u$ forms a contiguous subsequence of $(X_i)_{i\ge 1}$.
\end{enumerate}
The width of a path-decomposition of $G$ equals $\max_{i\ge 1}|X_i|-1$ and the pathwidth of $G$ is defined to be the minimum width among all path-decompositions.

Let $PW_2(m)$ denote the set of pathwidth-2 outerplanar triangulations on $m$ vertices. Let $\T(n)$ denote the set of trees on $n$ vertices.

In the next theorem, we bound $R_c(T_n,H_m)$ and $R_g(T_n,H_m)$ for any $T_n\in \T(n)$ and $H_m\in PW_2(m)$.

\begin{thm}\lab{thm:tree} Let $n\ge 1$ and $m\ge 1$.  For any $T_n\in\T(n)$ and $H_m\in PW_2(m)$:
\begin{enumerate}
\item[(i)] $R_c(T_n,H_m)\le (n-1)^2(m-1)+1$;
\item[(ii)] $R_g(T_n,H_m)=O(n^3m^2)$.
\end{enumerate}
\end{thm}

\noindent{\em Remark}: The bounds in Theorem~\ref{thm:tree} may be further improved if more is known about the structure of $T_n$. However, we are only interested in presenting bounds that are uniform for all trees.

 Note that every caterpillar has pathwidth at most 2 and it can be easily completed into a triangulation with pathwidth 2. So, Theorem~\ref{thm:tree} immediately implies that $R_g(T_n,T^C_m)=O(n^3m^2)$ for any $T_n\in \T(n)$ and any caterpillar $T^C_m$ on $m$ vertices. Theorem~\ref{thm:dumbbell}(ii) yields a better bound $n^2m^2$. However, this can be further improved because of the nice structure of caterpillars. We will prove the following bound.

\begin{thm}\lab{thm:tree-caterpillar} Let $n\ge 1$ and $m\ge 1$.  For any tree $T_n$ on $n$ vertices and any caterpillar $T^C_m$ on $m$ vertices, $R_g(T_n,T^C_m)=O(nm^2)$.
\end{thm}

Theorem~\ref{thm:tree-caterpillar} implies that $R_g(T_n)$ for any caterpillar $T_n$ on $n$ vertices is bounded by $O(n^3)$. In the next theorem, we prove another upper bound in terms of the maximum degree and the number of non-leaf vertices of $T_n$.

\begin{thm}\lab{thm:caterpillar}
For any caterpillar $T_n$ on $n$ vertices, $R_g(T_n)=O(\Delta \gamma n)$, where $\Delta$ is the maximum degree of $T_n$ and $\gamma$ is the number of non-leaf vertices of $T_n$.
\end{thm}

As an immediate corollary, we get a better (than $O(n^3)$) bound of $R_g(T_n)$, if $T_n$ is a caterpillar with small maximum degree or a small number of non-leaf vertices.

\begin{cor}
Suppose that $T_n$ is a caterpillar on $n$ vertices and suppose that $T_n$ has bounded maximum degree or $T_n$ has a bounded number of non-leaf vertices. Then, $R_g(T_n)=O(n^2)$.
\end{cor}

In the next theorem, we give a bound on $R_g(T_n)$, for a general tree $T_n$, in terms of the diameter of $T_n$ (the length of the longest path in $T_n$). We did not try to optimize the exponent in the upper bound. Note that whether there is a uniform polynomial bound of $R_g(T_n)$ for all $T_n\in\T(n)$ remains an open question.

\begin{thm}\lab{thm:cliqueTree}
For any tree $T_n$ on $n$ vertices,
$R_g(T_n)\le n^{d+1}$, where $d=d(T_n)$ denotes the diameter of $T_n$.
\end{thm}

\subsection{Note to the main results}

Trees have nice and simple structures. One would expect that they have small geometric Ramsey numbers (as for Ramsey numbers). But it seems that how easily we can embed a tree without edge-crossing depends heavily on certain parameters of the tree, e.g.\ the pathwidth or the diameter. To the best knowledge of the author, the graphs whose geometric Ramsey numbers have been proved to be polynomial all have pathwidth at most 2 (e.g. paths~\cite{KPTV}, cycles~\cite{KPTV,KR}, cycles triangulated from a single vertex~\cite{KPTV,KR}, matchings~\cite{KR}, outerplanar triangulations with pathwidth 2~\cite{CGKVV}). That may explain why it is challenging to obtain a good (polynomial) bound of the geometric Ramsey numbers uniformly for all trees.

\section{Proofs}
\subsection{Proof of Theorem~\ref{thm:dumbbell}(i,iii)}

To prove part (i), we will use the following unpublished result by Perles, stated in~\cite[Theorem D]{G}, whose proof can be found in~\cite{BKV}.

\begin{thm}[Perles]\lab{thm:Perles}
If a convex geometric graph on $N\ge n$ vertices has more than $\lfloor(n-2)N/2\rfloor$ edges, then it contains a non-crossing copy of any caterpillar with $n$ vertices.
\end{thm}

Let $K_P$ be a complete convex geometric graph on $N=(n-1)(m-1)+1$ vertices. If $K_P$ contains a blue copy of $K_m$, then by Theorem~\ref{thm:complete}, $K_P$ contains a blue non-crossing copy of $H_m$. So, we may assume that $K_P$ does not contain a blue copy of $K_m$ (a copy of $K_m$ whose edges are all coloured blue). Then, by Lemma~\ref{lem:Chvatal}, there is a subgraph $G$ of $K_P$ with minimum red degree (the number of incident red edges) at least $n-1$. Hence, $G$ contains at least $v(n-1)/2$ edges, where $v=|V(G)|\ge n$. By Theorem~\ref{thm:Perles}, $G$, and thus $K_P$, contains a red non-crossing copy of $T_n$. This proves that $R_c(T_n,H_m)\le N$. Part (i) then follows by~\eqn{lowerbound}.

Next, we prove part (iii).
Let $T_n$ be a tree on $n$ vertices which contains only two non-leaf vertices. Let $u$ and $v$ denote these two non-leaf vertices. Then $u$ and $v$ must be adjacent. Let $a$ and $b$ denote the degrees of $u$ and $v$ respectively. Then $n=a+b$. By~\eqn{lowerbound}, we only need to prove that $R_g(T_n,H_m)\le (n-1)(m-1)+1$.
Let $P$ be a set of $(n-1)(m-1)+1$ points in the plane in general position and consider any $2$-colouring of $K_P$. By Theorem~\ref{thm:complete}, we may assume that $K_P$ does not contain a blue copy of $K_m$. Then, by Lemma~\ref{lem:Chvatal}, there is a subgraph $S$ of $K_P$, whose minimum red degree is at least $n-1=a+b-1$. Let $u\in S$ be a point on the convex hull of $S$ and let $v$ be the point in $S$ such that $\{u,v\}$ is red and $u$ is incident with exactly $a-1$ red edges on one side (side A) of $h$, where $h$ is the line subtended from the line segment between $u$ and $v$. Thus, $u$ must be incident with at least $b-1$ red edges on the other side (side B). Since $v$ has red degree at least $a+b-1$, either it is incident with at least $a-1$ red edges on side A or at least $b-1$ red edges on side B. Either case would imply the existence of a red non-crossing copy of $T_n$.

\subsection{Proof of Theorems~\ref{thm:caterpillar} and~\ref{thm:dumbbell}(ii)}
\lab{sec:caterpillar}

In this section, we prove an upper bound for the geometric Ramsey numbers of caterpillars. Recently, Cibulka et al.~\cite{CGKVV} bounded the geometric Ramsey number of an outerplanar triangulation on $n$ vertices with pathwidth $2$ by $O(n^{10})$. This immediately gives an upper bound $O(n^{10})$ for the geometric Ramsey numbers of caterpillars. As we would expect, the bound for caterpillars should be significantly better than that. We will prove the bound in Theorem~\ref{thm:caterpillar}.

Let $G$ be a geometric graph on a set $P$ of points. Suppose that all points in $P$ are totally ordered, denoted by $p_1<p_2<\cdots<p_{|P|}$. We say that $p_{i_1}p_{i_2}\cdots p_{i_k}$ is a monotone path in $G$, if we have $p_{i_1}<p_{i_2}<\cdots < p_{i_k}$. Usually we will order the points in $P$ in a way so that a path in $G$ being monotone implies that it is non-crossing. For instance, we may order these points according to their $x$ or $y$-coordinates; or we may order the points in the clockwise or the counterclockwise order ``seen'' by a point not in $P$.

The following theorem (rephrased) was proved in~\cite[Theorem 4.3]{KPT}, by using a lemma of Dilworth~\cite{D}. We will use this theorem in the proof of Theorem~\ref{thm:caterpillar} and in the subsequent sections.

\begin{thm}\lab{thm:path-clique}
Let $P$ be a set of $(n-1)(m-1)+1\le mn$ ordered points. Then, for any $2$-edge-colouring of $K_P$, either there is a red monotone path $p_1< p_2<\cdots< p_{n}$ or there is a blue clique of order $m$.
\end{thm}

For a point $p$ in the plane, let $x(p)$ and $y(p)$ denote its $x$ and $y$ coordinates. Without loss of generality, for a finite set of points in the plane in general position, we may assume that their $x$ coordinates are all distinct and so are their $y$ coordinates. For two sets of points $A$ and $B$, we say $A\prec B$  if for any choice of $p\in A$ and $p'\in B$, $x(p)< x(p')$.

\begin{proof}[Proof of Theorem~\ref{thm:caterpillar}]
Let $P$ be a set of $4\Delta\gamma n$ points in general position in the plane. Consider any $2$-colouring of $K_P$. Let $\S_1,\ldots,\S_{2\gamma n}$ be a partition of $P$ such that they are separated by $2\gamma n-1$ vertical lines and each part contains exactly $2\Delta$ points. Inside each part $\S_i$, we colour a point red if it is incident with at least $\Delta$ red edges. Otherwise, it is incident with at least $\Delta$ blue edges and we colour it blue. We colour $\S_i$ red if there is a red point in $\S_i$. Otherwise, we colour $\S_i$ blue. Then, there is $c\in\{\mbox{blue,red}\}$ such that at least $\gamma n$ $\S_i$'s are coloured with the colour $c$. Without loss of generality, we may assume that $c$ is red.

Now let $\S_{i_1},\ldots, \S_{i_{\gamma n}}$ be $\gamma n$ red parts such that $\S_{i_1}\prec \S_{i_2}\prec \cdots\prec \S_{i_{\gamma n}}$. For every $1\le j\le \gamma n$, pick a point $u_j\in \S_{i_j}$ such that $u_j$ is incident with at least $\Delta$ points in $\S_{i_j}$ with a red edge. Then, by Theorem~\ref{thm:path-clique}, in the subgraph of $K_P$ induced by $\{u_j:\ 1\le j\le \gamma n \}$, either there is a red path $u_{j_1},\ldots,u_{j_{\gamma}}$ such that $x(u_{j_1})<\cdots<x(u_{j_{\gamma}})$, or a blue clique $K_n$. The first case implies the existence of a red non-crossing copy of $T_n$ (the fact of non-crossing follows by $x(u_{j_1})<\cdots<x(u_{j_{\gamma}})$ and that the red edges from $u_{j_i}$ to its neighbours inside $\S_{j_i}$ does not cross edges inside other $\S_{j'}$), whereas the second case implies the existence of a blue non-crossing copy of $T_n$ by Theorem~\ref{thm:complete}.
\end{proof}

\begin{proof}[Proof of Theorem~\ref{thm:dumbbell}(ii)]

The proof is very analogous to that of Theorem~\ref{thm:caterpillar}. Let $P$ be a set of $\Delta \gamma m^2$ points in the plane in general position and we may assume that there is no blue copy of $K_m$ in $K_P$ by Theorem~\ref{thm:complete}. Let $\S_1,\ldots, \S_{\gamma m}$ be a partition of $P$ such that these parts are separated by $\gamma m-1$ vertical lines and each part contains exactly $\Delta m$ points. By Theorem~\ref{thm:Chvatal}, there must be a red star of order $\Delta$ inside each $\S_i$. Let $u_i$ denote the central vertex of each of these red star in $\S_i$. By Theorem~\ref{thm:path-clique}, there must be a red monotone path $u_{i_1},\ldots,u_{i_{\gamma}}$ such that $x(u_{i_1})<\cdots<x(u_{i_{\gamma}})$. This red path together with all red edges $\{u_{i_j},v\}$ where $v\in \S_{i_j}$ forms a red non-crossing subgraph of $K_P$ which contains $T_n$ as a subgraph.
\end{proof}

\subsection{Proofs of Theorems~\ref{thm:tree} and~\ref{thm:tree-caterpillar}}
\lab{sec:caterpillar-tree}

To prove these two Theorems, we need to apply some results of mutually avoiding sets, defined as follows.

\begin{definition}
Two sets of points $A$ and $B$ in the plane are mutually avoiding if $|A|,|B|\ge 2$ and no line subtended by a pair of points in $A$ intersects the convex hull of $B$ and vice versa.
\end{definition}
Note that if $A\cup B$ are in convex position and $A$ and $B$ are separated by a straight line, then $A$ and $B$ are mutually avoiding.

\begin{thm}[Aronov et al.~\cite{AEGKKPS}]\lab{thm:avoiding}
If $A$ and $B$ are two sets of $n$ points, separated by a line, then there is a constant $c>0$ and two sets $A'\subset A$ and $B'\subset B$ such that $|A'|=|B'|=c\sqrt{n}$ and $A'$ and $B'$ are mutually avoiding.
\end{thm}

Let $A$ and $B$ be two disjoint sets of points in the plane. Define $K_{A,B}$ to be the complete geometric bipartite graph with the vertex set as points in $A\cup B$ and the edge set as the line segments between any two points $a\in A$ and $b\in B$. If $A$ and $B$ are mutually avoiding, then every vertex $a\in A$ ``sees'' vertices in $B$ clockwisely in the same order and vice versa.

Let $T_m$ be a caterpillar on $m$ vertices. Then, there is a unique partition $A\cup B$ of $V(T_m)$ such that both $A$ and $B$ induce an independent set. We say $T_m$ is $(m_1,m_2)$-decomposable, where $m_1+m_2=m$, if $|A|=m_1$ and $|B|=m_2$. The following observation follows easily from the above property of mutually avoiding sets.

\begin{obs}\lab{ob:avoiding} Let $T_m$ be a caterpillar on $m$ vertices. Suppose that $T_m$ is $(m_1,m_2)$-decomposable.
Let $A$ and $B$ be two mutually avoiding sets with $|A|=m_1$ and $|B|=m_2$. Then, $K_{A,B}$ contains a non-crossing copy of $T_m$.
\end{obs}

The pathwidth-2 outerplanar triangulations can be embedded in a similar way into a complete bipartite graph, as done in~\cite{CGKVV}, which we describe as follows. Recall that an outerplanar triangulation is a graph that has a planar embedding such that all its vertices are incident with the outer face and all inner faces have degree exactly 3. Recall also that $PW_2(n)$ denotes the set of pathwidth-2 outerplanar triangulations on $n$ vertices. It was proved in~\cite{CGKVV} that $PW_2(n)$ is equal to the following set of outerplanar triangulations $G$ on $n$ vertices.
\begin{equation}
\mbox{\{$G$: $\exists U\subseteq V(G)$ such that $U$ and $V(G)\setminus U$ each induces a path.  \}} \lab{PW2}
\end{equation}

For any $G\in PW_2(m)$, we say that $G$ is $(m_1,m_2)$-decomposable, where $m_1+m_2=m$, if there is $U\subset V(G)$ with $|U|=m_1$ such that $U$ and $V(G)\setminus U$ each induces a path. By~\eqn{PW2}, for every $G\in PW_2(m)$, there is $(m_1,m_2)$ such that $G$ is $(m_1,m_2)$-decomposable.

\begin{prop}\lab{p:H} Suppose that $H_m\in PW_2(m)$ is $(m_1,m_2)$-decomposable with $m_1+m_2=m$.
Let $A$ and $B$ be two mutually avoiding sets. Suppose $|A|=(m_1-1)(n-1)+1$ and $|B|=(m_2-1)(n-1)+1$ and all edges between $A$ and $B$ are coloured blue. Then, $K_{A\cup B}$ either contains a red clique $K_n$ or contains a blue non-crossing copy of $H_m$.
\end{prop}

\begin{proof} Since $A$ and $B$ are mutually avoiding, we can order points in $A$ as $u_1< u_2<\cdots$ and points in $B$ as $v_1< v_2<\cdots$, so that for each $u\in A$, $u$ sees $v_1,v_2$ in the clockwise order and for each $v\in B$, $v$ sees $u_1,u_2,\ldots$ in the counterclockwise order. Assume $K_{A\cup B}$ contains no red copies of $K_n$. Then, by Theorem~\ref{thm:path-clique}, there is a blue non-crossing monotone path $P_1$ on $m_1$ vertices in $A$ and a blue non-crossing monotone path $P_2$ on $m_2$ vertices in $B$. Since $P_1$ and $P_2$ are monotone, they each are non-crossing. Since $A$ and $B$ are mutually avoiding, the two paths $P_1$ and $P_2$, together with all blue edges between $A$ and $B$ contains a blue non-crossing embedding of $H_m$.
\end{proof}

Let $G$ be a geometric graph and let $T$ be a tree.
 Suppose that we want to prove that $G$ has a non-crossing embedding of $T$. It turns out that sometimes it easier to prove instead that $G$ contains a non-crossing copy of $T$ with some additional properties. Let $v$ be a vertex in $T$. Let $(T,v)$ denote the tree $T$ rooted at $v$.
 We say $G$ contains an {\em extreme} copy of $(T,v)$ with respect to the $x$ (or $y$) coordinate, if it contains a subgraph $G'$ and there is an isomorphism $f$ from $T$ to $G'$ such that $f(v)$ has the largest $x$ (or $y$) coordinate among all points $\{f(u):\ u\in T\}$. In most cases, instead of proving that $G$ contains a non-crossing copy of $T$, we prove the stronger statement that $G$ contains an extreme (with respect to the $x$ or $y$ coordinate) non-crossing copy of $T$.



\begin{proof}[Proof of Theorem~\ref{thm:tree-caterpillar}]
 Let $v\in T_n$ be an arbitrary vertex and denote by $(T_n,v)$ the tree $T_n$ rooted at $v$.  Let $c>0$ be a sufficiently large constant to be determined later. Let $P$ be a set of $c(n-1)m^2+1$ points in general position in the plane. We prove that for any $2$-colouring of $K_P$ that does not contain a blue non-crossing copy of $T^C_m$, it must contain an extreme red non-crossing copy of $(T_n,v)$ (with respect to the $y$ coordinate).

We prove this claim by induction on $n$. Clearly, $R_g(T_n,T^C_m)\le c(n-1)m^2+1$ for $n=1$. Now assume that $n\ge 2$ and assume that the assertion holds for every tree with less than $n$ vertices. For any $2$-colouring of $K_P$ without a blue non-crossing copy of $T^C_m$, we will prove that there is an extreme red non-crossing copy of $(T_n,v)$.  Let $u_1,\ldots,u_{\ell}$ denote the set of children of $T_n$. Let $(\widehat T_1,u_1),\ldots,(\widehat T_{\ell},u_{\ell})$ be the set of rooted trees created by removing $v$ from $T_n$ and rooting $u_1,\ldots,u_{\ell}$. Let $n_1,\ldots,n_{\ell}$ denote their orders.  Order points in $P$ so that $p_1< p_2$ if $y(p_1)<y(p_2)$. Suppose that $T^C_m$ is $(m_1,m_2)$-decomposable, for some $(m_1,m_2)$. Let $\S_0$ be the subset of $P$ containing the largest $a=cm_1^2\ell$ points in $P$. Let $h_1,\ldots,h_{\ell-1}$ be $\ell-1$ vertical lines that partition $P\setminus \S_0$ into $\ell$ parts, each containing at least $c(n_i-1)m^2+cm_2^2$ points respectively for $1\le i\le \ell$. This can be done since at most
\[
cm_1^2\ell +\sum_{i=1}^{\ell}\Big(c(n_i-1)m^2+cm_2^2\Big)\le c\ell m^2+cm^2(n-1-\ell)\le c(n-1)m^2+1
\]
total
points
are needed.
  Let $\S_1,\ldots,\S_{\ell}$ denote the $\ell$ parts. In each $\S_i$, let $\S'_i$ be a set of points defined repeatedly in the following way. Initially let $\S'_i=\emptyset$. Since $\S_i\setminus \S'_i$ contains at least $c(n_i-1)m^2+1$ points as long as $|\S'_i|\le cm_2^2-1$, by induction, there must be an extreme red non-crossing $(\widehat T_i,u_i)$ in $\S_i\setminus\S'_i$, since otherwise there will be a blue non-crossing $T^C_m$. Let $p$ be the root of this red copy of $(\widehat T_i,u_i)$ and add $p$ into $\S'_i$. Repeat this until there are exactly $cm_2^2$ points in $\S'_i$. We call these points in $\S'_i$ the {\em connectors}. Now we know that for each $i$ and for every connector $p\in\S'_i$, there is an extreme red non-crossing $(\widehat T_i,u_i)$ in $\S_i$ such that the root $u_i$ is at $p$. If there exists $p\in \S_0$, such that it is linked via a red edge to at least one connector in each $\S'_i$, then there is an extreme red non-crossing copy of $T_n$ with its root at $p$ and then we are done. Suppose this is not true. Then, for every $p\in \S_0$, there is $i$, such that $p$ is linked to all connectors in $\S'_i$ via blue edges. Since there are $cm_1^2\ell$ points in $\S_0$, by the pigeonhole principle, there must exist $\S'_0\subset \S_0$ and $i$ such that $|\S'_0|=cm_1^2$ and all edges between $\S'_0$ and $\S'_i$ are blue. Now $|\S'_0|=cm_1^2$ and $|\S'_i|=cm_2^2$. By Theorem~\ref{thm:avoiding}, provided that $c>0$ is sufficiently large, there are $\S''_0\subset \S'_0$ and $\S''_i\subset \S'_i$ such that $|\S''_0|=m_1$ and $|\S''_i|=m_2$ and $\S''_0$ and $\S''_i$ are mutually avoiding. Then by Observation~\ref{ob:avoiding}, there must be a blue non-crossing copy of $T^C_m$, contradicting our assumption. This completes our proof.

\end{proof}

The proof of Theorem~\ref{thm:tree} is very similar to that of Theorem~\ref{thm:tree-caterpillar}. We only sketch the proof.

\begin{proof}[Proof of Theorem~\ref{thm:tree}] Assume that $H_m$ is $(m_1,m_2)$-decomposable. The claim is trivially true for $n=1$. Now assume $n\ge 2$ and
we follow the same proof of Theorem~\ref{thm:tree-caterpillar}, except that we define $\S_0$ to be the largest $(m_1-1)(n-1)\ell+1$ points and the size of each $\S_i$ is at least $(n_i-1)^2(m-1)+(n-1)(m_2-1)+1$. This can be done since the total number of points needed is
\begin{eqnarray*}
&&\Big((m_1-1)(n-1)\ell+1\Big)+\sum_{i=1}^{\ell}\Big((n_i-1)^2(m-1)+(n-1)(m_2-1)+1\Big)\\
&&\hspace{0.5cm}\le\ell(n-1)(m-2)+1+\ell+(m-1)(n-1)\sum_{i=1}^{\ell}(n_i-1)\\
&&\hspace{0.5cm}
= (n-1)^2(m-1)+1-\ell(n-2)\le (n-1)^2(m-1)+1,
\end{eqnarray*}
since $n\ge 2$.
Assume $K_P$ does not contain a blue non-crossing copy of $H_m$. Then, similarly to the argument as in Theorem~\ref{thm:tree-caterpillar}, if there is no extreme red non-crossing copy of $(T_n,v)$, there must be $\S'_0\subset \S_0$ and $\S'_i\subset \S_i$ for some $1\le i\le \ell$ such that $|\S'_0|=(m_1-1)(n-1)+1$ and $|\S'_i|=(m_2-1)(n-1)+1$ such that all edges between $\S'_0$ and $\S'_i$ are coloured blue. By Proposition~\ref{p:H}, if there is no a blue non-crossing copy of $H_m$, there must be a red clique $K_n$, which implies the existence of an extreme red non-crossing copy of $(T_n,v)$ by Theorem~\ref{thm:complete}. This completes the proof of part (i). We can easily adapt this proof for part (ii).  We prove by induction on $n$, that $R_g(T_n,H_m)\le cm^2(n-1)^3+1$. We need to apply Theorem~\ref{thm:avoiding} to $\S'_0$ and $\S'_i$ and so we raise the sizes of $\S'_0$ and $\S'_i$ to $cm_1^2n^2$ and $cm_2^2n^2$ respectively, for some sufficiently large $c>0$. Thus, we will need to raise the size of $\S_0$ to $cm_1^2n^2\ell$ and raise the size of each $\S_i$ to $cm^2(n_i-1)^3+cm_2^2n^2$. This can be done as it is easy to verify that
\[
cm_1^2n^2\ell+\sum_{i=1}^{\ell}\Big(cm^2(n_i-1)^3+cm_2^2n^2\Big)\le c\ell n^2m^2+cm^2n^2\sum_{i=1}^{\ell}(n_i-1)\le cm^2(n-1)^3+1.
\]
The rest of the proof is the same as in part (i) except that we get $|\S'_0|=cm_1^2n^2$ and $|\S'_i|=cm_2^2n^2$. By Theorem~\ref{thm:avoiding}, we obtain mutually avoiding $\S''_0\subseteq \S'_0$ and $\S''_i\subseteq \S'_i$, with sizes $m_1n$ and $m_2n$ respectively. Then, similarly to part (i), there must be an extreme red non-crossing copy of $(T_n,v)$.
\end{proof}

\subsection{Proof of Theorem~\ref{thm:cliqueTree}}
\lab{sec:tree}

  For two rooted trees $(T_1,v_1)$ and $(T_2,v_2)$, we say $(T_1,v_1)$ is a sub-rooted-tree of $(T_2,v_2)$ if there is an injective homomorphism $f:V(T_1)\to V(T_2)$ with $f(v_1)=v_2$. Given a rooted tree $(T,v)$, the height of $(T,v)$ is the maximum distance from $v$ to $u$ among all $u\in T$. It is easy to see that there is $v\in T_n$ such that the height of $(T,v)$ is $h=\lceil d/2\rceil\le (d+1)/2$. Thus, it is sufficient to prove that for any rooted tree $(T_n,v)$ on $n$ vertices with height $h$, and for any $2$-colouring of $K_P$, where $|P|=n^{2h}$ and points in $P$ are in general position, $K_P$ contains either an extreme red non-crossing copy of $(T_n,v)$ or a blue $K_n$. Then our claim follows by Theorem~\ref{thm:complete}.

  We prove by induction on $h$. This claim is trivially true for $h=0$. Suppose $h\ge 1$ and assume that the claim holds for $h-1$. Let $P$ be a set of $n^{2h}$ points in the plane in general position. Consider any $2$-colouring of $K_P$ such that there are no blue copies of $K_n$.
Let $v_1,\ldots,v_k$ be the set of children of $v$ in $(T_n,v)$. Then, the removal of $v$ from $T$ results in $k$ new rooted trees with roots $v_1,\ldots,v_k$ respectively, for some $k<n$, and their total number of vertices is at most $n$. Clearly, there is a rooted tree $(T',u)$ with height $h-1$ and at most $n$ vertices, which contains each of these $k$ rooted trees as a sub-rooted-tree (such a tree $(T',u)$ can be obtained by simply merging the roots of each of these $k$ rooted trees into one).  Let $\S_1,\ldots,\S_{n^2}$ be a partition of $P$ such that each part contains exactly $n^{2(h-1)}$ points and these parts are separated by $n^2-1$ vertical lines. By induction and by our assumption that $K_P$ does not contain a blue copy of $K_n$, there is an extreme red non-crossing copy of $(T',u)$ in each of $\S_i$, $1\le i\le n^2$. Now let $P'=\{p_1,\ldots, p_{n^2}\}$ denote the set of $n^2$ points that are the roots of the red copies of $(T',u)$ in each $\S_i$. By Lemma~\ref{lem:Chvatal}, we may assume that there is a red star of size $n$ in $K_{P'}$, such that the non-leaf vertex of the star has the maximum $y$-coordinate among all vertices in the star. This red star together with the $n$ red copies of $T'$ form an extreme red rooted tree which is a supergraph of $T_n$. It is easy to see that this rooted tree is non-crossing by our construction. It implies that there is an extreme red non-crossing copy of $(T_n,v)$. This completes our proof.
\bigskip

\no {\bf Acknowledgement}: I would like to thank Gy. K\'{a}rolyi for making me aware of Perles' theorem, which leads to the proof of Theorem~\ref{thm:dumbbell}(i).


\begin{thebibliography}{99}

\bibitem{AEGKKPS}
B. Aronov, P. Erd\H{o}s, W. Goddard, D. J. Kleitman,
M. Klugerman, J. Pach, and L. J. Schulman, Crossing families, {\em Combinatorica}, 14(2): 127--134, 1994.

\bibitem{BKV} P. Brass, Gy. K\'{a}rolyi and P. Valtr, A Tur\'{a}n-type extremal theory of convex geometric graphs, {\em Discrete and Computational Geometry}, {\em Algorithms and Combinatorics}, 25, 275--300, 2003.

\bibitem{C} V. Chv\'{a}tal, Tree-complete graph ramsey numbers, {\em J. Graph Theorey} 1: 93, 1977.

\bibitem{CGKVV} J. Cibulka, P. Gao, M. Krc\'{a}l, T. Valla and P. Valtr, Polynomial bounds on geometric Ramsey numbers of ladder graphs, {\em Eurocomb 2013}.

\bibitem{D} R. P. Dilworth, A decomposition theorem for partially ordered sets,
{\em Annals of Mathematics}
2nd Series, 51(1): 161--166, 1950.

\bibitem{G} Gy. K\'{a}rolyi, Ramsey-Type Problems for Geometric Graphs, {\em Thirty Essays on Geometric Graph Theory}, 371--382, 2013.


\bibitem{GMPP} P. Gritzmann,
B. Mohar,  J. Pach and  R. Pollack, Embedding a planar triangulation with vertices at specified points, (Solution to problem E3341)
{\em  American Mathematical Monthly}, vol. 98, 165--166, 1991.

\bibitem{KPT} Gy. K\'{a}rolyi, J. Pach and G. T\'{o}th, Ramsey-Type Results for Geometric Graphs. I, {\em Discrete \& Computational Geometry } 18(3): 247--255, 1997.

\bibitem{KPTV} Gy. K\'{a}rolyi, J. Pach, G. T\'{o}th and P. Valtr, Ramsey-Type Results for Geometric Graphs. II, {\em Discrete \& Computational Geometry } 20(3): 375--388, 1998.


\bibitem{KR} Gy. K\'{a}rolyi and V. Rosta, Generalized and geometric Ramsey numbers for cycles, {\em Theor. Comput. Sci.} 263(1--2): 87--98, 2001.

\end{thebibliography}
\end{document}